\documentclass[final, 1p, times]{elsarticle}

\usepackage{amssymb}
\usepackage{amsthm}
\usepackage{amscd}
\usepackage{amsmath}
\usepackage{amsfonts}
\usepackage{graphicx}
\usepackage{color}
\usepackage{mathrsfs}
\usepackage{titletoc}
\usepackage{titlesec}
\newtheorem{theorem}{Theorem}
\newtheorem{lemma}{Lemma}

\begin{document}		
	\begin{frontmatter}
	\title{A Trudinger-Moser inequality involving $L^p$-norm on a closed Riemann surface}
	\author{Mengjie Zhang}
	\ead{zhangmengjie@ruc.edu.cn}
	\address{School of Mathematics, Renmin University of China, Beijing 100872, P.R.China}		
	\begin{abstract}
In this paper, using the method of blow-up analysis, we obtained a Trudinger-Moser inequality involving $L^p$-norm on a closed Riemann surface and proved the existence of an extremal function for the corresponding Trudinger-Moser functional.
This extends an early result of Yang \cite{Yang2007}. Similar result in the Euclidean space was also established by Lu-Yang \cite{Lu-Y2009}.
\end{abstract}
		
	\begin{keyword}
	Trudinger-Moser inequality, Riemann surface, blow-up analysis, extremal function.\\
    2010 MSC: 46E35; 58J05.
	\end{keyword}
		
	\end{frontmatter}

\section{Introduction}
Let $\Omega\subseteq \mathbb{R}^2$ be a smooth bounded domain and $W_0^{1,2}(\Omega)$ be the completion of $C_0^{\infty}(\Omega)$ under the Sobolev norm
$\|\nabla_{\mathbb{R}^2} u\|_2=( \int_{\Omega}{|\nabla_{\mathbb{R}^2} u|^2}dx ) ^{{1}/{2}},$
where $\nabla_{\mathbb{R}^2}$ is the gradient operator on ${\mathbb{R}^2}$ and $\|\cdot\|_2$ denotes the standard $L^2$-norm.
The classical Trudinger-Moser inequality \cite{Yudovich, Pohozaev, Peetre, Trudinger1967, Moser1970}, as a limit case of the Sobolev embedding, says
    \begin{equation}\label{Trudinger-Moser}
    \sup_{u\in W_0^{1, 2}(\Omega), \, \|\nabla_{\mathbb{R}^2} u\|_2\leq 1}
    \int_\Omega e^{\ \beta u^2}dx<+\infty, \ \forall \ \beta\leq 4\pi.
    \end{equation}
Moreover, $4\pi$ is called the best
constant for this inequality in the sense that when $\beta> 4\pi$, all integrals in (\ref{Trudinger-Moser}) are still finite, but the supremum is infinite.
It is interesting to know whether or not the supremum in (\ref{Trudinger-Moser}) can be attained.
For this topic, we refer the reader to Carleson-Chang \cite{C-C}, Flucher \cite{Flucher}, Lin \cite{Lin}, Adimurthi-Struwe \cite{A-Struwe}, Li \cite{Li-JPDE,Li-Sci}, Yang \cite{Yang-IJM} and Lu-Yang  \cite{Lu-Y2009}.

There are many extensions of (\ref{Trudinger-Moser}).
Adimurthi-Druet \cite{AD} generalized (\ref{Trudinger-Moser}) to the following form:
    \begin{equation}\label{T-AM}
  \sup _ { u \in W _ { 0 } ^ { 1,2 } ( \Omega ) , \| \nabla_{\mathbb{R}^2} u \| _ { 2 } \leq 1 } \int _ { \Omega } e ^ { 4 \pi u ^ { 2 } \left( 1 + \alpha \| u \| _ { 2 } ^ { 2 } \right) } dx <+\infty, \ \forall\ 0 \leq \alpha < \lambda _ { 1 } ( \Omega ),
\end{equation}
where $\lambda _ { 1 } ( \Omega )$ is the first eigenvalue of the Laplacian with Dirichlet boundary condition in $\Omega .$
This inequality is sharp in the sense that if $\alpha \geq \lambda _ { 1 } ( \Omega )$, all integrals in (\ref{T-AM}) are still finite, but the supremum is infinite. Obviously, (\ref{T-AM}) is reduced to (\ref{Trudinger-Moser}) when $\alpha=0$.
Various extensions of the inequality (\ref{T-AM}) were obtained by Yang \cite{Yang-JFA-06,Yang-JDE-15}, Tintarev \cite{Tintarev} and Zhu \cite{ZhuJY} respectively.
It was extended by Lu-Yang \cite{Lu-Y2009} to a version involving $L ^ { p }$-norm for any $p > 1$, namely
\begin{equation}\label{T-LY}
\sup _ { u \in W _ { 0 } ^ { 1 ,2} ( \Omega ) , \| \nabla_{\mathbb{R}^2} u \| _ { 2 } \leq 1 } \int _ { \Omega } e ^ { 4 \pi \left( 1 + \alpha \| u \| _ { p } ^ { 2 } \right) u ^ { 2 } } d x < + \infty, \ \forall\ 0 \leq \alpha < \lambda _ { p } ( \Omega ),
\end{equation}
where $\lambda _ { p } ( \Omega ) = \inf _ { u \in W _ { 0 } ^ { 1,2 } ( \Omega ) , u \not\equiv 0 } { \| \nabla_{\mathbb{R}^2} u \| _ { 2 } ^ { 2 } } /{ \| u \| _ { p } ^ { 2 } }$. This inequality is sharp in the sense that all integrals in (\ref{T-LY}) are still finite when $ \alpha \geq \lambda_p ( \Omega)$, but the supremum is infinite.
Moreover, for sufficiently small $\alpha> 0$, the supremum is attained.
Analogs of (\ref{T-LY}) are naturally expected for the cases of
do \'O-de Souza \cite{do-de2014,do-de2016}, Nguyen \cite{N2017,N2018}, Li \cite{Lixiaomeng}, Li-Yang \cite{L-Y}, Zhu \cite{Zhu}, Fang-Zhang \cite{F-Z} and Yang-Zhu \cite{Yang-Zhu2018,Yang-Zhu2019}.

Trudinger-Moser inequalities were introduced on Riemannian manifolds by Aubin \cite{A}, Cherrier \cite{C} and Fontana \cite{Fontana}. In particular, let $( \Sigma ,\ g )$ be a closed Riemann surface, $W^ { 1 , 2 } ( \Sigma,\ g) $ the completion of $C ^ { \infty } ( \Sigma)$ under the norm
$\|u \| ^2_ { W^{1,2}( \Sigma,\ g) } =  \int _ { \Sigma }( u^2+| \nabla_g u | ^2)\, dv_g$, where $\nabla_g$ stands for the gradient operator on $(\Sigma,\ g)$.
Denote
$$
\mathcal { H } = \left\{ u \in W ^ { 1,2 } ( \Sigma,\ g) : \| \nabla_g u \| _ { 2 } \leq 1 , \int _\Sigma  {u}\ dv_g = 0 \right\}.
$$
Then there holds
\begin{equation}\label{T-M}
  \sup _ { u \in \mathcal { H } }\int _ { \Sigma } e ^ {\ \beta u^ 2} dv_g<+\infty, \ \forall \ \beta\leq 4\pi.
    \end{equation}
Moreover, $4\pi$ is called the best
constant for this inequality in the sense that when $\beta> 4\pi$, all integrals in (\ref{T-M}) are still finite, but the supremum is infinite.
Based on the works of Ding-Jost-Li-Wang \cite{DJLW} and Adimurthi-Struwe \cite{A-Struwe}, Li \cite{Li-JPDE,Li-Sci} proved the existence of extremals for the supremum in (\ref{T-M}).
Later Yang \cite{Yang2007} obtained the same inequality as (\ref{T-AM}) on a closed Riemann surface:
\begin{equation}\label{T-M1}
\sup _ { u \in \mathcal { H } }\int _ { \Sigma} e ^ {4\pi u^ 2(1+\alpha\parallel u\parallel_2^2) } dv_g<+\infty, \ \forall\ 0 \leq \alpha < \lambda _ { 1 } ( \Sigma),
\end{equation}
where $\lambda _ { 1 } ( \Sigma)$ is the first eigenvalue of the Laplace-Beltrami operator with respect to the metric $g$.
This inequality is sharp in the sense that if $\alpha \geq \lambda _ { 1 } ( \Sigma)$, all integrals in (\ref{T-M1}) are still finite, but the supremum is infinite. Furthermore, for sufficiently small $\alpha > 0 $, the supremum in (\ref{T-M1}) is attained.\\

In this paper, in view of (\ref{T-LY}), we expect (\ref{T-M1}) with $\|u\|_2$ replaced by $\|u\|_p$ for any $p > 1$. Our main result is now stated as follows:
\begin{theorem}\label{T1}
     Let $(\Sigma,g)$ be a closed Riemann surface. For any real number $p>1$, we set
\begin{equation}\label{la}
\lambda_{ p} = \inf _ { u \in W^ { 1,2 } ( \Sigma,\ g) , \int _ { \Sigma } u\, dv_g = 0 , u \not\equiv 0 } \frac { \| \nabla_g u \| _ { 2 } ^ { 2 } } { \| u \| _ { p } ^ { 2 } }
\end{equation}
and
\begin{equation*}
J^\alpha_{\beta}(u) = \int _ { \Sigma } e ^ {\ \beta u^ 2(1+\alpha\parallel u\parallel_p^2) } dv_g.
\end{equation*}
Then we have\\
($i$) for any $\alpha \geq \lambda_{ p},$ $ \sup _ { u \in \mathcal { H } } J^\alpha_{4\pi}(u) = +\infty$;\\
($ii$) for any $0 \leq \alpha < \lambda_{ p},$ $ \sup _ { u \in \mathcal { H } } J^\alpha_{4\pi}(u) < +\infty$;\\
($iii$) for sufficiently small $\alpha> 0$, there exists a function $u_\alpha \in \mathcal{H}$ such that $\sup _ { u \in \mathcal { H } } J^\alpha_{4\pi}(u) =J^\alpha_{4\pi}(u_\alpha).$
     \end{theorem}

Following the lines of Li \cite{Li-JPDE}, Yang \cite{Yang2007} and Lu-Yang \cite{Lu-Y2009}, we prove Theorem \ref{T1} by using the method of blow-up analysis.
The remaining part of this note is organized as follows:
In Section 2, we prove (Theorem \ref{T1}, ($i$)) by constructing test functions.
In Section 3, we prove (Theorem \ref{T1}, ($ii$)) in three steps:
firstly, the existence of maximizers for subcritical functionals $\sup _ { u \in \mathcal { H } }J^\alpha_{4\pi-\epsilon}(u)$ is proved, and the corresponding Euler-Lagrange equation is given;
secondly, the asymptotic behavior of maximizer is given by blow-up analysis;
finally, under the assumption that blow-up occurs, we deduce an upper bound of $\sup _ { u \in \mathcal { H } }J^\alpha_{4\pi}(u)$.
In Section 4, we construct a sequence of functions to show (Theorem \ref{T1}, ($iii$)) holds.

\section{Proof of Part ($i$)}
In this section, we select test functions to prove Part ($i$).
Let $\lambda_{ p}$ be defined by (\ref{la}) and $\alpha \geq \lambda_{ p}$.
Following Yang \cite{Yang2007}, we get $\lambda_{ p}$ is attained by some function $v _ { 0 } \in W^ { 1,2 } ( \Sigma,\ g) \cap C ^ { \infty } ( \Sigma)$ satisfying
$\int_{\Sigma}v_0\ dv_g=0$ and
\begin{equation}\label{a1}
\lambda_{ p} \lVert v_0 \rVert _{p}^{2}=\lVert \nabla_g v_0\lVert _{2}^{2}=1.
\end{equation}
Consequently, there exist a point $ x_0\in \Sigma$ with $v _ { 0 } ( x_0 ) > 0 $ and
a neighborhood $U$ of $x_0$ with $v_ { 0 }(x) \geq v _ { 0 } ( x_0 ) / 2$ in $U$ .
We take an isothermal coordinate system $(\phi^{-1}(\mathbb{B}_\delta),\phi)$ such that $\phi(x_0) = 0$ and $\phi^{-1}(\mathbb{B}_\delta)\subset U $. In such coordinates, the metric $g$ has the representation $g = e^{2f} (dx_1^2 +dx_2^2 )$ and $f$ is a smooth function with $f(0) = 0$.

Let $\delta = { t ^{-1}_ { \epsilon } \sqrt { -\log  { \epsilon } } }$, where $t _ { \epsilon } > 0$ such that $-t _ { \epsilon } ^ { 2 } \log  { \epsilon } \rightarrow + \infty$ and $t _ { \epsilon } ^ { 2 } \sqrt { -\log { \epsilon } } \rightarrow 0$ as $\epsilon\rightarrow0$.
We choose a continuous cut-off function $\eta(x)$,
which is equal to $0$ in $\phi ^{-1}\left( \mathbb{B}_{\delta} \right)$ and equal to $1$ in $\Sigma \backslash \phi ^{-1}\left( \mathbb{B}_{2\delta} \right)$ such that $ \eta(x)\in C ^ { \infty } ( \Sigma)$ and $| \nabla \eta | \leq 2 / \delta$.
Choosing some $x_\delta\in\mathbb{B}_\delta$ with $\left| x _ { \delta } \right| = \delta$, we set
\begin{equation*}
\bar{v}_{\epsilon}\left( x \right) =	\left\{
\linespread{1.5}\selectfont
\begin{aligned}
	&\sqrt{\frac{-\log{\epsilon}}{2\pi}},\quad \ \ \ \ \ \ \ \ \ \ \ \ \ \ \ \ \ &|x|<\epsilon,\ \ \ \ \ \ \\
	&\frac{\sqrt{{\frac{-\log{\epsilon}}{2\pi}}\left( \log \delta -\log|x| \right) -t_{\epsilon}\bar{v}_0\left( x_{\delta} \right) \left( \log \epsilon -\log|x| \right)}}{\log \delta -\log \epsilon},\quad &\epsilon \le |x|\le \delta,\\
\end{aligned} \right.
\end{equation*}
and
\begin{equation}\label{a3}
  v_{\epsilon}\left( x \right) =\left\{ \begin{aligned}
	&\ \bar{v}_{\epsilon}\circ\phi,\quad &\phi ^{-1}\left( \mathbb{B}_{\delta} \right), \ \ \ \ \\
	&\ t_{\epsilon}\left[\bar{ v}_0\left( x_{\delta} \right) +\eta \left( x \right) \left( v_0\left( x \right) -\bar{v}_0\left( x_{\delta} \right) \right) \right] ,\ &\Sigma \backslash \phi ^{-1}\left( \mathbb{B}_{\delta} \right),\\
\end{aligned} \right.
\end{equation}
where $\bar{v}_0=v_0\circ\phi^{-1}$.
Taking $v^*_{\epsilon}=v_{\epsilon}-\int_{\Sigma}v_{\epsilon}\,dv_g/{\rm Area}( \Sigma)$ such that $\int_{\Sigma}v_{\epsilon}^*\,dv_g=0$,
we obtain
\begin{equation}\label{a4}
\begin{aligned}
\int_{\Sigma}{\left| \nabla_g v_{\epsilon}^{*} \right|}^2dv_g
&=\int_{\phi ^{-1}\left( \mathbb{B}_{\delta} \right)}{\left| \nabla_g v_{\epsilon} \right|}^2dv_g+\int_{\phi ^{-1}\left( \mathbb{B}_{2\delta} \right) \backslash \phi ^{-1}\left( \mathbb{B}_{\delta} \right)}{\left| \nabla_g v_{\epsilon} \right|}^2dv_g+\int_{\Sigma \backslash \phi ^{-1}\left( \mathbb{B}_{2\delta} \right)}{\left| \nabla_g v_{\epsilon} \right|}^2dv_g
\\
&=\frac{2\pi \left( t_{\epsilon}\bar{v}_0\left( x_{\delta} \right) -\sqrt{\frac{-\log\epsilon}{2\pi}} \right)^2}{\log \delta -\log \epsilon}+t_{\epsilon}^{2}O( \delta ^2 ) +t_{\epsilon}^{2}\left( 1+O( \delta ^2 ) \right)
\\
&=1-2\sqrt{\frac{-2\pi}{\log\epsilon}}
t_{\epsilon}\bar{v}_0\left( x_{\delta} \right) \left( 1+o\left( 1 \right) \right) +t_{\epsilon}^{2}\left( 1+O( \delta ^2 ) \right).
\end{aligned}
\end{equation}
Setting $m_\epsilon=v_{\epsilon}^{*} /\lVert v_{\epsilon}^{*} \rVert _{2}^{2}\in\mathcal{H}$. According to (\ref{a1})-(\ref{a4}), we get
\begin{equation*}
  \begin{aligned}
	\lambda_{ p} \lVert m_{\epsilon} \rVert _{p}^{2}
& \ge \frac{\lambda_{ p}}{\lVert \nabla_g v_{\epsilon}^{*} \rVert _{2}^{2}}\left( \left( \int_{\Sigma \backslash \phi ^{-1}\left( \mathbb{B}_{2\delta} \right)}{v_{\epsilon}^{p}}dv_g \right) ^{\text{2/}p}+O( \delta ^2 ) \right)\\
	&=\frac{\lambda_{ p} t_{\epsilon}^{2}}{\lVert \nabla_g v_{\epsilon}^{*} \rVert _{2}^{2}}\left( \lVert v_0 \rVert _{p}^{2}+O( \delta ^2 ) \right)\\
	&=t_{\epsilon}^{2}\left(  \lambda_{ p} \lVert v_0 \rVert _{p}^{2}+O( \delta ^2 ) \right) \left( 1+O( t_{\epsilon}^{2} ) \right)\\
	&=t_{\epsilon}^{2}\left( 1+O( t_{\epsilon}^{2}) +O( \delta ^2 ) \right).\\
\end{aligned}
\end{equation*}
On $\phi ^{-1}\left( \mathbb{B}_{\epsilon} \right)$, combining $\bar{v} _ { 0 } \left( x _ { \delta } \right) = \bar{v} _ { 0 } ( 0 ) + o ( 1 )$ and $-t _ { \epsilon } ^ { 2 } \log{ \epsilon }\ O ( \delta ^ { 2 }) = O ( 1 )$, we have
\begin{equation*}\label{a5}
  \begin{aligned}
  4\pi \left( 1+\lambda_{ p} \lVert m_{\epsilon} \rVert _{p}^{2} \right) m_{\epsilon}^{2}=
  &4\pi \left( 1+\lambda_{ p} \lVert m_{\epsilon} \rVert _{p}^{2} \right) \frac{v_{\epsilon}^{2}+O(\delta ^2)}{\lVert \nabla_g v_{\epsilon}^{*} \rVert _{2}^{2}}\\
	\ge& \left( -\text{2}\log \epsilon +O(\delta ^2) \right) \left( 1+t_{\epsilon}^{2}\left( 1+O( t_{\epsilon}^{2}) +O(\delta ^2) \right) \right)\\
&\times \left( 1+2\sqrt{\frac{-2\pi}{\log \epsilon}}t_{\epsilon}\bar{v}_0\left(x_\delta \right) \left( 1+o\left( 1 \right) \right) -t_{\epsilon}^{2}\left( 1+O(\delta ^2) \right) \right)\\
=&-2\log \epsilon+ Ct_{\epsilon}\sqrt{- \log \epsilon}\,\bar{v}_0\left(0 \right) \left( 1+o\left( 1 \right) \right)+O\left( 1 \right).
\end{aligned}
\end{equation*}
Hence there holds
$$
\int_{\Sigma}{e}^{4\pi \left( 1+\lambda_{ p} \lVert m_{\epsilon} \rVert _{p}^{2} \right) m_{\epsilon}^{2}}dv_g
\ge Ce^{t_{\epsilon}\sqrt{-\log \epsilon}\,\bar{v}_0\left( 0 \right) \left( 1+o\left( 1 \right) \right)}\rightarrow +\infty
$$
as $\epsilon \rightarrow 0 $, where $C>0$ is independent of $\epsilon$.
Obviously, $ \sup _ { u \in \mathcal { H } } J^\alpha_{4\pi}(u) \ge \lim_{\epsilon \rightarrow 0}J^\alpha_{4\pi}( m_{\epsilon})=+\infty$ when $\alpha \geq \lambda_{ p}$.
	\section{Proof of Part ($ii$)}
Under the condition of $0 \leq \alpha < \lambda_{ p}$,
we first prove the existence of maximizers for subcritical functionals $\sup _ { u \in \mathcal { H } }J^\alpha_{4\pi-\epsilon}(u)$ and give the corresponding Euler-Lagrange equation.
Then we give the asymptotic behavior of maximizer by means of blow-up analysis.
Finally, we derive an upper bound of $\sup _ { u \in \mathcal { H } }J^\alpha_{4\pi}(u)$ for the occurrence of blow-up.

	\subsection{Existence of maximizers for subcritical functionals.}
\begin{lemma}\label{L1}
 Let $0 \leq \alpha < \lambda_{ p}$. For any $\epsilon > 0$, there exists some function $u_\epsilon \in \mathcal{H}  \cap C^{\infty} ( \Sigma)$
such that
\begin{equation*}\label{2}
  J^\alpha_{4\pi-\epsilon}(u_\epsilon)=\sup _ { u \in \mathcal { H } } J^\alpha_{4\pi-\epsilon}(u).
\end{equation*}
\end{lemma}

Since the proof of Lemma \ref{L1} is completely analogous to that of ({\cite{Yang2007}}, Lemma 3.2), we omit it in this paper.
Moreover, we derive the Euler-Lagrange equation of $u_{\epsilon}$ by a direct calculation, namely
\begin{equation}\label{e-u}
	\left\{
\linespread{1.5}\selectfont
\begin{aligned}
&{\Delta_g u _ { \epsilon } = \frac { \beta _ { c } } { \lambda _ { \epsilon } } u _ { \epsilon } e ^ { \alpha _ { \epsilon } u _ { \epsilon } ^ { 2 } } + \gamma _ { \epsilon }\|u _ { \epsilon }\|^{2-p}_p|u _ { \epsilon }|^{p-2} u _ { \epsilon } - \frac { \mu _ { \epsilon } } { \lambda _ { \epsilon } } } ,\\
&{ \int _ { \Sigma } u _ { \epsilon }\, dv_g = 0 , \quad \left\| \nabla_g u _ { \epsilon } \right\| _ { 2 } = 1 } ,\\
&{ \alpha _ { \epsilon } = ( 4 \pi - \epsilon ) \left( 1 + \alpha \left\| u _ { \epsilon } \right\| _ { p } ^ { 2 } \right) } ,\\
&{ \beta _ { \epsilon } = \left( 1 + \alpha \left\| u _ { \epsilon } \right\| _ { p } ^ { 2 } \right) / \left( 1 + 2 \alpha \left\| u _ { \epsilon } \right\| _ { p } ^ { 2 } \right) } ,\\
&{ \gamma _ { \epsilon } = \alpha / \left( 1 + 2 \alpha \left\| u _ { \epsilon } \right\| _ { p } ^ { 2 } \right) } ,\ \ { \lambda _ { \epsilon } = \int _ { \Sigma } u _ { \epsilon } ^ { 2 } e ^ { \alpha _ { \epsilon } u _ { \epsilon } ^ { 2 } } dv_g }, \\
&{ \mu _ { \epsilon } = \left(\beta _ { \epsilon } \int _ { \Sigma } u _ { \epsilon } e ^ { \alpha _ { \epsilon } u _ { \epsilon } ^ { 2 } } d V _g+\lambda_\epsilon\gamma_\epsilon\|u_\epsilon\|^{2-p}_p\int _ { \Sigma } |u_\epsilon|^{p-2}u _ { \epsilon }\, dv_g\right)/{\rm Area}( \Sigma) }, \end{aligned} \right.
	\end{equation}
where $\Delta_g$ denotes the Laplace-Beltrami operator.
From Lebesgue's dominated convergence theorem and the nature of supremum, it is easy to get that
    \begin{equation}\label{3}
    {\lim_{\epsilon \rightarrow 0}}\ J^ \alpha_{4\pi-\epsilon}(u)=J^\alpha_{4\pi}(u).
    \end{equation}

	\begin{lemma}\label{L2}
    Let $\lambda _{\epsilon}$ be defined by (\ref{e-u}), then we have $\liminf_{\epsilon \rightarrow 0} \ \lambda _{\epsilon}>0.$
    \end{lemma}
	
	\begin{proof}
	Using the fact $e^t\leqslant 1+te^t$ for any $t\geqslant 0 $, we have
$$
{ \lambda _ { \epsilon } = \int _ { \Sigma } u _ { \epsilon } ^ { 2 } e ^ { \alpha _ { \epsilon } u _ { \epsilon } ^ { 2 } } dv_g }
\geqslant \frac{1}{\alpha_\epsilon}\int _ { \Sigma }( e ^ { \alpha _ { \epsilon } u _ { \epsilon } ^ { 2 } }-1 )\ dv_g.
$$
Together with the fact that $\alpha_\epsilon$ is bounded, (\ref{2}) and (\ref{3}), we imply the lemma.
\end{proof}

\begin{lemma}
$\mu _ { \epsilon } / \lambda _ { \epsilon }$ is bounded.
\end{lemma}

\begin{proof}
According to (\ref{e-u}), the fact of $\beta _ { \epsilon } \leq 1$ and Lemma \ref{L2}, we have
\begin{equation*}
  \begin{aligned}
  \left| \frac{\mu _{\epsilon}}{\lambda _{\epsilon}} \right|
&\le \frac{1}{{\rm Area}\left( \Sigma \right)}\left( \frac{\beta _{\epsilon}}{\lambda _{\epsilon}}\int_{\left| u_{\epsilon} \right|\ge 1}{u}_{\epsilon}e^{\alpha _{\epsilon}u_{\epsilon}^{2}}dv_g+\frac{\beta _{\epsilon}}{\lambda _{\epsilon}}\int_{\left| u_{\epsilon} \right|<1}{u}_{\epsilon}e^{\alpha _{\epsilon}u_{\epsilon}^{2}}dv_g+\gamma _{\epsilon}\|u_{\epsilon}\|_{p}^{2-p}\int_{\Sigma}{|}u_{\epsilon}|^{p-2}u_{\epsilon}\,dv_g \right)\\
&\le \frac{1}{{\rm Area}\left( \Sigma \right)}\left( 1+\frac{e^{\alpha _{\epsilon}}}{\lambda _{\epsilon}}{\rm Area}\left( \Sigma \right) +\gamma _{\epsilon}\|u_{\epsilon}\|_{p}^{2-p}\int_{\Sigma}{|}u_{\epsilon}|^{p-2}u_{\epsilon}\,dv_g \right) \le C.
  \end{aligned}
\end{equation*}
\end{proof}

\subsection{Blow-up analysis}
We now perform the blow-up analysis on $u_\epsilon$.
Here and in the sequel, we do not distinguish sequence and subsequence.
Since $u_{\epsilon}\in \mathcal{H} $,
there exists some function $u_0$ such that $u_{\epsilon}\rightharpoonup u_0$ weakly in
$W^{1, 2}( \Sigma ,\ g)$ and $u_{\epsilon}\rightarrow u_0$ in $ L^q( \Sigma,\ g)$ as $\epsilon\rightarrow0$.
With no loss of generality, we assume in the following, $c_{\epsilon}=|u_{\epsilon}\left( x_\epsilon \right)| ={\max}_{\Sigma}\ u_{\epsilon}\rightarrow+\infty$ and $x_\epsilon\rightarrow x_0$ as $\epsilon\rightarrow0$.

\begin{lemma}\label{L3}
There hold $u_0=0$ and $ |\nabla_g u_{\epsilon}|^2dv_g\rightharpoonup \delta _{x_0} $ in sense of measure, where $\delta _{x_0}$ is the usual Dirac measure centered at $x_0$.
\end{lemma}

\begin{proof}
Suppose $u_0 \not\equiv 0$, we can see that $0< \eta:=\lVert \nabla_g u_0 \rVert _{2}^{2}\leq1$. Hence $\lVert \nabla_g \left( u_{\epsilon}-u_0 \right) \rVert _{2}^{2}\rightarrow 1-\eta<1$ and $1 + \alpha \left\| u _ { \epsilon } \right\| _ { p } ^ { 2 } \rightarrow 1 + \alpha \left\| u _ { 0 } \right\| _ { 2 } ^ { 2 } \leq 1 + \left\| \nabla_g u _ { 0 } \right\| _p^ { 2 }=1+\eta$.
For sufficiently small $\epsilon$, we obtain $(1 + \alpha \left\| u _ { \epsilon } \right\| _ { p } ^ { 2 } )\lVert \nabla_g \left( u_{\epsilon}-u_0 \right) \rVert _{2}^{2}\leqslant  (2-\eta^2)/2<1$.
By the H\"older inequality, there holds
\begin{equation*}
\begin{aligned}
\int_{\Sigma}{e^{q\alpha _{\epsilon}u_{\epsilon}^{2}}}dv_g&\leqslant \int_{\Sigma}{e^{q\alpha _{\epsilon}\left( 1+1/\delta \right) u_{0}^{2}+q\alpha _{\epsilon}\left( 1+\delta \right) \left( u_{\epsilon}-u_0 \right) ^2}}dv_g\\
	&\leqslant \left( \int_{\Sigma}{e^{rq\alpha _{\epsilon}\left( 1+1/\delta \right) u_{0}^{2}}}dv_g \right) ^{\frac{1}{r}}\left( \int_{\Sigma}{e^{sq\left( 1+\delta \right) \left( 4\pi -\epsilon \right) \left( 1+\alpha \lVert u_{\epsilon}\lVert _{p}^{2} \right) \lVert \nabla_g \left( u_{\epsilon}-u_0 \right) \rVert _{2}^{2}\frac{\left( u_{\epsilon}-u_0 \right) ^2}{\lVert \nabla_g \left( u_{\epsilon}-u_0 \right) \rVert _{2}^{2}}}}dv_g \right) ^{\frac{1}{s}}\\
	&\leqslant C\left( \int_{\Sigma}{e^{sq\left( 1+\delta \right) \left( 4\pi -\epsilon \right)\left(\frac{2-\eta}{2}\right) \frac{\left( u_{\epsilon}-u_0 \right) ^2}{\lVert \nabla_g \left( u_{\epsilon}-u_0 \right) \rVert _{2}^{2}}}}dv_g \right) ^{\frac{1}{s}}
\end{aligned}
\end{equation*}
for $\delta >0$, $r,\ s,\ q>1$ satisfying ${1}/{r}+{1}/{s}=1$ and $sq\left( 1+\delta \right) \left( 2-\eta \right) <2 $. From the Trudinger-Moser inequality (\ref{T-M}), we get $e^{\alpha _{\epsilon}u_{\epsilon}^{2}}$ is bounded in $L^q\left( \Sigma,\ g \right) $. Hence $\Delta_g u_{\epsilon}$ is bounded in some $L^q\left( \Sigma,\ g \right)$ from (\ref{e-u}). Applying the elliptic estimate to (\ref{e-u}), one gets $u_{\epsilon}$ is uniformly bounded, which contradicts $c_{\epsilon}\rightarrow +\infty$. Therefore, the assumption is not established.

Suppose $|\nabla_g u_{\epsilon}|^2dv_g\rightharpoonup \mu \ne \delta_{x_0}$ in sense of measure. There exists some $r_0>0$ such that ${\lim}_{\epsilon \rightarrow 0}\int_{B_{r_0}(x_0)}{\left| \nabla_g u_{\epsilon} \right|}^2dv_g=\eta <1$. For sufficiently small $\epsilon$, we can see that $\int_{B_{r_0}(x_0)}{\left| \nabla_g u_{\epsilon} \right|}^2dv_g\leqslant (\eta +1)/{2}<1$.
We take an isothermal coordinate system $(U,\phi)$ near $x_0$ such that $\phi(x_0) = 0$. In such coordinates, the metric $g$ has the representation $g = e^{2f} (dx_1^2 +dx_2^2 )$ and $f$ is a smooth function with $f(0) = 0$. Denote $\bar{u}_\epsilon=u_\epsilon\circ\phi^{-1}$.
Then we choose a cut-off function $\psi$ in $C_{0}^{1}\left(\phi\left(B_{r_0}(x_0)\right) \right)$, which is equal to $1$ in $\phi\left(B_{r_0/2}(x_0)\right) $  such that
$\lVert \nabla_g \left( \psi \bar{u}_{\epsilon} \right) \lVert _{2}^{2}
\leqslant (\eta +3)/{4}<1.$
Hence we obtain
\begin{equation*}
\begin{aligned}
\int_{ B_{r_0/2}\left( x_0 \right) }{e^{\alpha _{\epsilon}qu_{\epsilon}^{2}}}dv_g&=\int_{\phi \left( B_{r_0/2}\left( x_0 \right) \right)}{e^{\alpha _{\epsilon}q\bar{u}_{\epsilon}^{2}}}e^{2f}dx
\\
&\leq C\int_{\phi \left( B_{r_0}\left( x_0 \right) \right)}{e^{\alpha _{\epsilon}q\left( \psi \bar{u}_{\epsilon} \right) ^2}}dx
\\
&\leq C\int_{\phi \left( B_{r_0}\left( x_0 \right) \right)}{e^{\alpha _{\epsilon}q  \frac{\eta +3}{4}  \frac{\left( \psi \bar{u}_{\epsilon} \right) ^2}{\lVert \nabla_g \left( \psi \bar{u}_{\epsilon} \right) \lVert _{2}^{2} }}}dx.
\end{aligned}
\end{equation*}
From the Trudinger-Moser inequality (\ref{T-M}),  we get $e^{\alpha _{\epsilon}u_{\epsilon}^{2}}$ is bounded in $L^q\left( B_{r_0/2}(x_0),\ g\right) $ for any $q>1$ satisfying $q(\eta +3)/4\leq1$. Applying the elliptic estimate to (\ref{e-u}), one gets $u_{\epsilon}$ is uniformly bounded in $B_{r_0/4}\left( x_0 \right) $, which contradicts $c_{\epsilon} \rightarrow +\infty$. Therefore, Lemma \ref{L3} follows.
	\end{proof}

Now we analyse the asymptotic behavior of $u_\epsilon$ near the concentration point $x_0$.

\begin{lemma}\label{L4}
Let
  \begin{equation}\label{r}
  r _ { \epsilon } ^ { 2 } = \frac { \lambda _ { \epsilon } } { \beta _ { \epsilon } c _ { \epsilon } ^ { 2 }  e ^ { \alpha _ { \epsilon } c _ { \epsilon } ^ { 2 } }}.
  \end{equation}
Then for any positive integer $k$, there holds $\lim_{\epsilon\rightarrow0}r_\epsilon^2c_\epsilon^k= 0$.
\end{lemma}

\begin{proof}
Using the $\mathrm{H\ddot{o}lder}$ inequality, (\ref{T-M1}) and (\ref{e-u}), we have
	\begin{equation*}
	  \begin{aligned}
r_{\epsilon}^{2}c_{\epsilon}^{k}&=\frac{\lambda _{\epsilon}}{\beta _{\epsilon}c_{\epsilon}^{2-k}e^{\alpha _{\epsilon}c_{\epsilon}^{2}}}
\le \frac{\int_{\Sigma}{u_{\epsilon}^{k}e^{\left( 1-\delta \right) \alpha _{\epsilon}u_{\epsilon}^{2}}dv_g}}{\beta _{\epsilon}e^{\left( 1-\delta \right) \alpha _{\epsilon}c_{\epsilon}^{2}}}
\\
&\le \frac{1}{\beta _{\epsilon}e^{\left( 1-\delta \right) \alpha _{\epsilon}c_{\epsilon}^{2}}}\left( \int_{\Sigma}{u_{\epsilon}^{kr}dv_g} \right) ^{\frac{1}{r}}\left( \int_{\Sigma}{e^{\left( 1-\delta \right) s\alpha _{\epsilon}u_{\epsilon}^{2}}dv_g} \right) ^{\frac{1}{s}}\\
&\le \frac{C}{\beta _{\epsilon}e^{\left( 1-\delta \right) \alpha _{\epsilon}c_{\epsilon}^{2}}}\left( \int_{\Sigma}{u_{\epsilon}^{kr}dv_g} \right) ^{\frac{1}{r}}
      \end{aligned}
	\end{equation*}
for any $0<\delta <1$, $1/r+1/s=1$ and some constant $C$ depending on $\delta$ and $s$.
Then the lemma follows from Lemma \ref{L3} immediately.
\end{proof}

We take an isothermal coordinate system $(U,\phi)$ near $x_0$ such that $\phi(x_0) = 0$. In such coordinates, the metric $g$ has the representation $g = e^{2f} (dx_1^2 +dx_2^2 )$ and $f$ is a smooth function with $f(0) = 0$. Denote $\bar{u}_\epsilon=u_\epsilon\circ\phi^{-1}$, $\bar{x}_\epsilon=\phi(x_\epsilon)$ and $U _ { \epsilon } = \{ x \in \mathbb { R } ^ { 2 } : \bar{x} _ { \epsilon } + r _ { \epsilon } x \in \phi\,(U ) \}$. Then $U_ { \epsilon }\rightarrow \mathbb{R}^2\,\mathrm{as}\, \epsilon \rightarrow 0$ from Lemma \ref{L4}.
Define two blowing up functions in $ U _ { \epsilon }$
\begin{equation}\label{psi}
\psi _ { \epsilon } ( x ) = \frac{\bar{u}_\epsilon \left(\bar{ x} _ { \epsilon } + r _ { \epsilon } x \right)} {c _ { \epsilon }}
\end{equation}
and
\begin{equation}\label{phi}
\varphi _{\epsilon}\left( x \right) =c_{\epsilon}\left( \bar{u}_{\epsilon}\left( \bar{x}_{\epsilon}+r_{\epsilon}x \right) -c_{\epsilon} \right).
\end{equation}
By (\ref{e-u}) and (\ref{r})-(\ref{phi}), a direct computation shows
	\begin{equation}\label{e-psi}
\begin{aligned}
   -\Delta _{\mathbb{R}^2}\psi _{\epsilon}=\left( c_{\epsilon}^{-2}\psi _{\epsilon}e^{\alpha _{\epsilon}\left( \psi _{\epsilon}+1 \right) \varphi _{\epsilon}}+c_{\epsilon}^{p-2}r_{\epsilon}^{2}\gamma _{\epsilon}\lVert u_{\epsilon}\lVert _{p}^{2-p}|\psi _{\epsilon}|^{p-2}\psi _{\epsilon}-\frac{r_{\epsilon}^{2}\mu _e}{c_{\epsilon}\lambda _{\epsilon}} \right) e^{2f\left( \bar{x}_{\epsilon}+r_{\epsilon}x \right)}
    \end{aligned}
	\end{equation}
and
	\begin{equation}\label{e-phi}
\begin{aligned}
-\Delta_{\mathbb{R}^2}\varphi _{\epsilon}=\left( \psi _{\epsilon}e^{\alpha _e\varphi _{\epsilon}\left( 1+\psi _{\epsilon} \right)}+c_{\epsilon}^{p}r_{\epsilon}^{2}\gamma _{\epsilon}\lVert u_{\epsilon}\lVert _{p}^{2-p}|\psi _{\epsilon}|^{p-2}\psi _{\epsilon}-\frac{c_{\epsilon}r_{\epsilon}^{2}\mu _{\epsilon}}{\lambda _{\epsilon}} \right) e^{2f\left( \bar{x}_{\epsilon}+r_{\epsilon}x \right)},
    \end{aligned}
	\end{equation}	
where $\Delta_ { \mathbb { R } ^ { 2 } }$ denotes the Laplace operator on $\mathbb { R } ^ { 2 }$.

Noticing (\ref{e-psi}), (\ref{e-phi}) and using the same argument as \cite{Yang2007}, we lead to
	\begin{equation}\label{5}
    \lim_{ \epsilon\rightarrow0}\psi _{\epsilon}= 1\,\,\mathrm{in}\, \, C_{loc}^{1}( \mathbb{R}^2)
	\end{equation}
and
	\begin{equation}\label{6}
    \lim_{ \epsilon\rightarrow0}\varphi _{\epsilon}= \varphi \, \, \mathrm{in}\, \, C_{loc}^{1}( \mathbb{R}^2),
	\end{equation}
where $\varphi$ satisfies $-\Delta_{\mathbb{R}^2} \varphi =e^{8\pi  \varphi}$ and
$\varphi \left( 0 \right) =0=\sup_{\mathbb{R}^2} \varphi$.
By the standard uniqueness result of the ordinary differential equation, there holds
    \begin{equation}\label{7}
    \varphi \left( x \right) =-\frac{1}{4\pi}\log ( 1+{\pi}|x|^{2}).
    \end{equation}
It follows that
	\begin{eqnarray}\label{8}
	\int_{\mathbb{R}^2}{e^{8\pi \left( 1+\beta \right) \varphi}|x|^{2\beta}dx=1}.
	\end{eqnarray}

Next we discuss the convergence behavior of $u_\epsilon$ away from $x_0$.
Denote $u_{\epsilon,\,\beta }=\min  \{\ \beta c_{\epsilon} , u_{\epsilon} \}  \in W^{1,2}\left(\Sigma,\ g \right) $ for any real number $0<\beta<1$. Following (\cite{Yang2007}, Lemma 4.5), we get
\begin{equation}\label{9}
\lim_{\epsilon \rightarrow 0}\left\| \nabla_g u _ { \epsilon ,\ \beta } \right\| _ { 2 } ^ { 2 }=\beta.
\end{equation}

	\begin{lemma}\label{L6}
 Letting $\lambda _{\epsilon}$ be defined by (\ref{e-u}), we obtain
\begin{equation*}
\begin{aligned}
\limsup _ { \epsilon \rightarrow 0 } J _ { 4 \pi - \epsilon } ^ { \alpha } \left( u _ { \epsilon } \right) = {\rm Area} (\Sigma)+\underset{\epsilon \rightarrow 0}{\lim}\,\frac{\lambda _{\epsilon}}{c_{\epsilon}^{2}}
  ={\rm Area}( \Sigma)+\lim_{R\rightarrow +\infty}\lim_{\epsilon \rightarrow 0}\int_{\phi ^{-1}\left( \mathbb{B}_{Rr_{\epsilon}}\left( \bar{x}_{\epsilon} \right) \right)}{e^{\alpha _{\epsilon}u_{\epsilon}^{2}}}\,dv_g.
  \end{aligned}
\end{equation*}
	\end{lemma}
	
	\begin{proof}
Recalling (\ref{e-u}) and (\ref{9}), for any real number $0<\beta<1$, one gets
	\begin{equation*}
\begin{aligned}
J _ { 4 \pi - \epsilon } ^ { \alpha } \left( u _ { \epsilon } \right)-{\rm Area}( \Sigma)
&=
\int_{\left\{ x\in \Sigma :\, u_{\epsilon}\le \beta c_{\epsilon} \right\}}{( e^{\alpha _{\epsilon}u_{\epsilon}^{2}}-1 )\, dv_g}+\int_{\left\{ x\in \Sigma :\, u_{\epsilon}>\beta c_{\epsilon} \right\}}{( e^{\alpha _{\epsilon}u_{\epsilon}^{2}}-1 )\, dv_g}
\\
&\le \int_{\Sigma}{( e^{\alpha _{\epsilon}u_{\epsilon ,\,\beta}^{2}}-1 )\, dv_g}+\frac{u_{\epsilon}^{2}}{\beta ^2c_{\epsilon}^{2}}\int_{\left\{ x\in \Sigma :\,u_{\epsilon}>\beta c_{\epsilon} \right\}}{e^{\alpha _{\epsilon}u_{\epsilon}^{2}}\,dv_g}
\\
&\le \int_{\Sigma}e^{\alpha _{\epsilon}u_{\epsilon ,\,\beta}^{2}}\alpha _{\epsilon}u_{\epsilon}^{2}\,dv_g+\frac{\lambda _{\epsilon}}{\beta ^2c_{\epsilon}^{2}}
\\
&\le \left( \int_{\Sigma}e^{r\alpha _{\epsilon}u_{\epsilon ,\,\beta}^{2}}dv_g \right) ^{1/r}\left( \int_{\Sigma} \alpha _{\epsilon}^su_{\epsilon}^{2s}\,dv_g \right) ^{1/s}+\frac{\lambda _{\epsilon}}{\beta ^2c_{\epsilon}^{2}}.
\end{aligned}
	\end{equation*}
By (\ref{T-M1}) and (\ref{9}), $e^{\alpha _{\epsilon}u_{\epsilon ,\,\beta}^{2}}$ is bounded in $L^r\left( \Sigma,\ g \right)$ for some $r>1$.
Then letting $\epsilon \rightarrow 0$ first and then $\beta \rightarrow 1$, we obtain
	$$\limsup _ { \epsilon \rightarrow 0 } J _ { 4 \pi - \epsilon } ^ { \alpha } \left( u _ { \epsilon } \right)-{\rm Area}( \Sigma)\leq \underset{\epsilon \rightarrow 0}{\limsup}\,\frac{\lambda _{\epsilon}}{c_{\epsilon}^{2}}.$$
According to (\ref{e-u}), we have $$
\limsup _ { \epsilon \rightarrow 0 }J_{4\pi -\epsilon}^{\alpha}\left( u_{\epsilon} \right) -{\rm Area}\left( \Sigma \right) \ge \liminf _ { \epsilon \rightarrow 0 }\int_{\Sigma}{\frac{u_{\epsilon}^{2}}{c_{\epsilon}^{2}}e^{\alpha _{\epsilon}u_{\epsilon}^{2}}}\,dv_g-\lim_ { \epsilon \rightarrow 0 }\frac{1}{c_{\epsilon}^{2}}\int_{\Sigma}{u_{\epsilon}^{2}}\,dv_g=\liminf _ {\epsilon\rightarrow0}\frac{\lambda_{\epsilon}}{c_{\epsilon}^{2}}.
$$
Then there holds $$\limsup _ { \epsilon \rightarrow 0 } J _ { 4 \pi - \epsilon } ^ { \alpha } \left( u _ { \epsilon } \right) = {\rm Area}( \Sigma)+\underset{\epsilon \rightarrow 0}{\lim}\,\frac{\lambda _{\epsilon}}{c_{\epsilon}^{2}}.$$

Applying (\ref{e-u}) and (\ref{r})-(\ref{phi}), one has
\begin{equation*}
  \begin{aligned}
\int_{\phi ^{-1}( \mathbb{B}_{Rr_{\epsilon}}( \bar{x}_{\epsilon} ) )}{e^{\alpha _{\epsilon}u_{\epsilon}^{2}}}dv_g
&=\int_{\mathbb{B}_{Rr_{\epsilon}}\left( \bar{x}_{\epsilon} \right)}{e^{\alpha _{\epsilon}u_{\epsilon}^{2}\left( x \right)}e^{2f\left( x \right)}}dx
\\
&=\int_{\mathbb{B}_R\left( 0 \right)}{r_{\epsilon}^{2}e^{\alpha _{\epsilon}c_{\epsilon}^{2}\left( x \right)}e^{\alpha _{\epsilon}\left( \psi _{\epsilon}\left( x \right) +1 \right) \varphi _{\epsilon}\left( x \right)}e^{2f\left( \bar{x}_{\epsilon}+r_{\epsilon}x \right)}}dx
\\
&=\int_{\mathbb{B}_R\left( 0 \right)}\frac{\lambda _{\epsilon}}{\beta _{\epsilon}c_{\epsilon}^{2}}{e^{\alpha _{\epsilon}\left( \psi _{\epsilon}\left( x \right) +1 \right)\, \varphi _{\epsilon}\left( x \right)}e^{2f\left( \bar{x}_{\epsilon}+r_{\epsilon}x \right)}}dx.
  \end{aligned}
\end{equation*}
Combining (\ref{5})-(\ref{8}), we get
\begin{equation*}
\lim_{R\rightarrow +\infty}\lim_{\epsilon \rightarrow 0}\int_{\phi ^{-1}( \mathbb{B}_{Rr_{\epsilon}}( \bar{x}_{\epsilon} ) )}{e^{\alpha _{\epsilon}u_{\epsilon}^{2}}}dv_g=\lim_{\epsilon \rightarrow 0}\frac{\lambda _{\epsilon}}{c_{\epsilon}^{2}}.
\end{equation*}
Summarizing, we lead to the lemma.
    \end{proof}

Next we consider the properties of $c_{\epsilon}u_{\epsilon}$.
Using the similar idea of (\cite{Yang2007}, Lemma 4.7), we get
\begin{equation}\label{10}
 \frac { \beta _ { \epsilon } } { \lambda _ { \epsilon } } c _ { \epsilon } u _ { \epsilon } e ^ { \alpha _ { \epsilon } u _ { \epsilon } ^ { 2 } } dv_g \rightharpoonup \delta_{x_0}.
\end{equation}
We also need the following result belonging to \cite{Yang2007}:
\begin{lemma}\label{L5}
  Assume $u \in C ^ { \infty } ( \Sigma)$ is a solution of $\Delta_g u = f ( x )$ in $(\Sigma,\ g)$ and satisfies $ { \| u \| _ { 1 } \leq c _ { 0 } \| f \| _ { 1 } }$. Then for any $1 < q < 2 $, there holds $\| \nabla_g u \| _ { q } \leq C \left( q , c _ { 0 } , \Sigma,\ g \right) \| f \| _ { 1 }.$
\end{lemma}

\begin{lemma}\label{LG}
  For any $1 < q < 2$, $c _ { \epsilon } u _ { \epsilon }$ is bounded in $W ^ { 1 , q } ( \Sigma,\ g)$.
  Moreover, there holds
	\begin{equation*}
	\left\{ \begin{aligned}
	&c_{\epsilon}u_{\epsilon}\rightharpoonup G\, \, \mathrm{weakly\, \, in\,}\, W^{1,q}\left( \Sigma,\ g \right), \,\forall 1<q<2, \\
	&c_{\epsilon}u_{\epsilon}\rightarrow G\, \, \mathrm{strongly\, \, in  \,}\, L^s\left( \Sigma,\ g \right), \,\forall 1<s<\frac{2q}{2-q}, \\
	&c_{\epsilon}u_{\epsilon}\rightarrow G\ \mathrm{in\,}\, {C_{loc}^{1}\left( \Sigma\backslash \left\{ x_0 \right\} \right) },
	\end{aligned} \right.
	\end{equation*}
where $G$ is a Green function satisfying $\int_{\Sigma}{Gdv_g}=0$ and
\begin{equation}\label{e-G}
\Delta_g G=\delta _{x_0}+\alpha \lVert G \rVert _{p}^{2-p}\left| G \right|^{p-2}G-\frac{1}{{\rm Area}\left( \Sigma \right)}\left( 1+\alpha \lVert G \rVert _{p}^{2-p}\int_{\Sigma}{\left| G \right|}^{p-2}G\,dv_g \right).
\end{equation}

\end{lemma}

\begin{proof}
 By (\ref{e-u}), there holds
 	\begin{equation}\label{11}
\Delta_g \left( c_{\epsilon}u_{\epsilon} \right) =\frac{\beta _{\epsilon}}{\lambda _{\epsilon}}c_{\epsilon}u_{\epsilon}e^{\alpha _{\epsilon}u_{\epsilon}^{2}}+\gamma _{\epsilon}\lVert c_{\epsilon}u_{\epsilon} \rVert _{p}^{2-p}\left| c_{\epsilon}u_{\epsilon} \right|^{p-2}c_{\epsilon}u_{\epsilon}-c_{\epsilon}\frac{\mu _{\epsilon}}{\lambda _{\epsilon}}.
	\end{equation}
According to (\ref{e-u}), (\ref{10}) and $\mathrm{H\ddot{o}lder}$ inequality, we have
\begin{equation}\label{12}
\begin{aligned}
\left| \frac{c_{\epsilon}\mu _{\epsilon}}{\lambda _{\epsilon}} \right|&=\frac{1}{{\rm Area}\left( \Sigma \right)}\left( \int_{\Sigma}{\frac{c_{\epsilon}\beta _{\epsilon}}{\lambda _{\epsilon}}}u_{\epsilon}e^{\alpha _{\epsilon}u_{\epsilon}^{2}}dv_g+c_{\epsilon}\gamma _{\epsilon}\lVert u_{\epsilon} \rVert _{p}^{2-p}\int_{\Sigma}{|}u_{\epsilon}|^{p-2}u_{\epsilon}\,dv_g \right)
\\
&\le \frac{1}{{\rm Area}\left( \Sigma \right)}\left( 1+o_{\epsilon}\left( 1 \right) +\gamma _{\epsilon}\lVert c_{\epsilon}u_{\epsilon} \rVert _{p}^{2-p}\int_{\Sigma}{|c_{\epsilon}}u_{\epsilon}|^{p-2}c_{\epsilon}u_{\epsilon}\,dv_g \right)
\\
&\le \frac{1}{{\rm Area}\left( \Sigma \right)}\left( 1+o_{\epsilon}\left( 1 \right) +\lVert c_{\epsilon}u_{\epsilon} \rVert _p\gamma _{\epsilon}{\rm Area}\left( \Sigma \right) ^{\frac{1}{p}} \right).
\end{aligned}
\end{equation}

We claim that $\left\| c _ { \epsilon } u _ { \epsilon } \right\| _ { p }$ is bounded. Suppose not, we can assume  $\left\| c _ { \epsilon } u _ { \epsilon } \right\| _ { p } \rightarrow + \infty$ as $\epsilon \rightarrow 0$. Let $w _ { \epsilon } = c _ { \epsilon } u _ { \epsilon } / \left\| c _ { \epsilon } u _ { \epsilon } \right\| _ { p }\in W^{1,2}( \Sigma,\ g)$. Then $\left\| w _ { \epsilon } \right\| _ { p } = 1$, $\int _ { \Sigma } w _ { \epsilon } dv_g = 0$ and
\begin{equation}\label{13}
  \Delta_g w_{\epsilon}=\frac{1}{\lVert c_{\epsilon}u_{\epsilon} \rVert _p}\frac{\beta _{\epsilon}}{\lambda _{\epsilon}}c_{\epsilon}u_{\epsilon}e^{\alpha _{\epsilon}u_{\epsilon}^{2}}+\gamma _{\epsilon}\left| w_{\epsilon} \right|^{p-2}w_{\epsilon}-\frac{1}{\lVert c_{\epsilon}u_{\epsilon} \rVert _p}\frac{c_{\epsilon}\mu _{\epsilon}}{\lambda _{\epsilon}}.
\end{equation}
From (\ref{12}), one gets
$$
\left| \frac{1}{\lVert c_{\epsilon}u_{\epsilon} \rVert _p}\frac{c_{\epsilon}\mu _{\epsilon}}{\lambda _{\epsilon}} \right|\le \alpha {\rm Area}\left( \Sigma \right) ^{\frac{1-p}{p}}+o_{\epsilon}\left( 1 \right).
$$
Then we obtain $\Delta_g w _ { \epsilon }$ is bounded in $L^1 ( \Sigma,\ g)$.
According to Lemma \ref{L5}, we have $w_\epsilon$ is bounded in $W^{1,q}( \Sigma,\ g)$ for any $1<q<2$. Then there exists $w$ satisfying $w _ { \epsilon } \rightharpoonup w$  weakly in $W^{ 1,\,q}( \Sigma,\ g)$ and $w _ { \epsilon } \rightarrow w$ strongly in $L^s( \Sigma,\ g)$ for any $1<s<{2q}/(2-q)$.
Testing (\ref{13}) with $\phi \in C^1 ( \Sigma)$, letting ${ \epsilon } \rightarrow 0$ and combining (\ref{10}), we get
\begin{equation}\label{14}
\int_{\Sigma}\nabla_g\phi \nabla_g w\,dv_g=\alpha \int_{\Sigma}{\phi}\left| w \right|^{p-2}w\,dv_g- \frac{\alpha\int_{\Sigma}{\left| w \right|}^{p-2}w\,dv_g}{{\rm Area}\left( \Sigma \right)}\int_{\Sigma}{\phi}\,dv_g.
\end{equation}
Since $0\leq\alpha < \lambda_{ p}$ and (\ref{14}), one derives that $w \equiv 0$, which contradicts the fact of $\| w \| _ { p } = 1$. Hence $\left\| c _ { \epsilon } u _ { \epsilon } \right\| _ { p }$ is bounded.

According to this and (\ref{12}), we have $\Delta_g \left( c _ { \epsilon } u _ { \epsilon } \right)$ is bounded in $L^1( \Sigma,\ g)$.
Again by Lemma \ref{L5}, there holds $c _ { \epsilon } u _ { \epsilon }$ is bounded in $W ^ { 1 , q } ( \Sigma,\ g)$ for any $1 < q < 2$.
It is easy to get $c_{\epsilon}u_{\epsilon}\rightharpoonup G\, \, \mathrm{weakly\, \, in\,}\, W^{1,q}\left( \Sigma,\ g \right)$ for any $1<q<2$ and $c_{\epsilon}u_{\epsilon}\rightarrow G\, \, \mathrm{strongly\, \, in  \,}\, L^s\left( \Sigma,\ g \right)$ for any $1<s<{2q}/{(2-q)}$.

We choose a cut-off function $\rho$ in $C^{\infty}\left( \Sigma \right)$, which is equal to $0$ in $B_\delta(x_0)$ and equal to $1$ in $\Sigma\backslash B_{2\delta}(x_0)$ such that
$\lim_{\epsilon\rightarrow0}{\lVert \nabla_g \left( \rho u_{\epsilon} \right) \rVert _{2}^{2}}=0.$
Hence there holds
\begin{equation*}
  \int_{\Sigma \backslash B_{2\delta}\left( x_0 \right)}{e^{s\alpha _{\epsilon}u_{\epsilon}^{2}}}dx\leqslant \int_{\Sigma \backslash B_{2\delta}\left( x_0 \right)}{e^{s\alpha _{\epsilon}\lVert \nabla_g \left( \rho u_{\epsilon} \right) \rVert _{2}^{2}\frac{\rho ^2u_{\epsilon}^{2}}{\lVert \nabla_g \left( \rho u_{\epsilon} \right) \rVert _{2}^{2}}}}dx.
\end{equation*}
From the Trudinger-Moser inequality (\ref{T-M}),  ${e^{\alpha _{\epsilon}u_{\epsilon}^{2}}}$ is bounded in $L^s\left( \Sigma,\ g \right) $ for some $s>1$.
Applying the elliptic estimate and the compact embedding theorem to (\ref{11}), we obtain $c_{\epsilon}u_{\epsilon}\rightarrow G$ in $ C_{loc}^{1}\left( {\Sigma }\backslash \left\{ x_0 \right\} \right).$
Testing (\ref{11}) by $\phi \in C^1\left( \Sigma \right)$, we lead to (\ref{e-G}) follows.
	\end{proof}

Applying the elliptic estimate, we can decompose $G$ as the form:
\begin{equation}\label{G}
  G=-\frac{1}{2\pi}\log \left| r \right|+A_{x_0}+\sigma(x),
\end{equation}
where $r=dist(x,x_0)$, $A_{x_0}$ is a constant and $\sigma(x)\in C^{\infty}( \Sigma)$ with $\sigma(x_0)=0$.

\subsection{Upper bound estimate}
	\begin{lemma}\label{L8}
Under the hypotheses $c_{\epsilon}\rightarrow +\infty$ and $x _ { \epsilon } \rightarrow x _ { 0 } \in \Sigma$, there holds
	\begin{equation*}\label{15}
	\sup _ { u \in \mathcal { H } } J _ { 4 \pi } ^ { \alpha } ( u ) \leq {\rm Area}( \Sigma) + \pi e ^ { 1 + 4 \pi A _ { x_0} }.
    \end{equation*}
	\end{lemma}

	\begin{proof}
To derive an upper bound of $\sup _ { u \in \mathcal { H } } J _ { 4 \pi } ^ { \alpha } ( u )$, we use the capacity estimate, which was first used by Li \cite{Li-JPDE} in this topic. We take an isothermal coordinate system $(U,\phi)$ near $x_0$ such that $\phi(x_0) = 0$. In such coordinates, the metric $g$ has the representation $g = e^{2f} (dx_1^2 +dx_2^2 )$ and $f$ is a smooth function with $f(0) = 0$. Denote $\bar{u}_\epsilon=u_\epsilon\circ\phi^{-1}$.

We claim that
\begin{equation}\label{16}
\underset{\epsilon \rightarrow 0}{\lim}\ \frac{\lambda _{\epsilon}}{c_{\epsilon}^{2}}\leq \pi e^{1+4\pi A_{x_0}}.
\end{equation}
To confirm this claim, we set
$W_{a,b}:=\left\{ \bar{u}_\epsilon\in W^{1,2}\left(\mathbb{B}_\delta \setminus \mathbb{B}_{Rr_\epsilon}\right)  :\ \bar{u}_\epsilon({\delta})=a, \, \bar{u}_\epsilon\left({Rr_\epsilon}\right)=b \right\}$ for some small $\delta\in(0,1)$ and some fixed $R>0$.
According to (\ref{6}), (\ref{7}), (\ref{G}) and Lemma \ref{LG}, one gets
 $$
 a=\frac{1}{c_{\epsilon}}\left( \frac{1}{2\pi}\log \frac{1}{\delta}+A_{x_0}+o_\delta(1)+o_{\epsilon}(1)\right)$$ and
$$
b=c_{\epsilon}+\frac{1}{c_{\epsilon}}\left( -\frac{1}{4\pi}\log ( 1+{\pi}R^{2} ) +o_{\epsilon}\left( 1 \right) \right),$$
where $o_\delta(1)\rightarrow0$, $o_{\epsilon}(1)\rightarrow0$ as $\epsilon\rightarrow 0$ .
From a direct computation, there holds
\begin{equation}\label{17}
2\pi \left( a-b \right) ^2=2\pi c_{\epsilon}^{2}+\text{2}\log \delta -4\pi A_{x_0}-\log ( 1+\pi R^2 ) +o_{\delta}\left( 1 \right) +o_{\epsilon}\left( 1 \right).
\end{equation}
From the direct method of variation, it is easy to see that
$\inf_{u\in W_{a,b}} \int_{{{Rr_\epsilon}}\leq |x| \leq {\delta}}{|\nabla_{\mathbb{R}^2} u|^2}dx$ can be attained by $m(x)\in W_{a,b}$ with $\Delta_{\mathbb{R}^2} m\left( x \right) =0$.
We can check that
\begin{equation*}
m\left( x \right) =\frac{a\left( \log|x|-\log \left( {Rr_\epsilon} \right) \right) +b\left( \log \delta -\log|x| \right)}{\log \delta -\log \left( {Rr_\epsilon} \right)}
\end{equation*}
and
\begin{equation}\label{19}
\int_{{{Rr_\epsilon}}\leq |x| \leq {\delta}}{|\nabla_{\mathbb{R}^2} m\left( x \right) |^2}dx=\frac{2\pi \left( a-b\right) ^2}{\log \delta -\log \left( {Rr_\epsilon} \right)}.
\end{equation}
Recalling (\ref{e-u}) and (\ref{r}), we have
\begin{equation}\label{20}
\log \delta -\log \left( {Rr_\epsilon} \right) =\log \delta -\log R-\frac{1}{2}\log \frac{\lambda _{\epsilon}}{c_{\epsilon}^{2}}+{2\pi c_{\epsilon}^{2} }+o_\epsilon(1).
\end{equation}
Letting $u^*_{\epsilon}=\max \left\{ a,\ \min \left\{b,\ \bar{u}_{\epsilon} \right\} \right\} \in W_{a,b}$, one gets $\left| \nabla_{\mathbb{R}^2} u^*_{\epsilon} \right|\leq \left| \nabla_{\mathbb{R}^2} \bar{u}_{\epsilon} \right|$ in $\mathbb{B}_{\delta}\backslash \mathbb{B}_{Rr_{\epsilon}^{\text{1/}\left( 1+\beta \right)}}$ for sufficiently small $\epsilon$.
According to these and Lemma 1, we obtain
\begin{equation}\label{21}
  \begin{aligned}
 \int_{Rr_{\epsilon}\le |x|\le \delta}{|\nabla _{\mathbb{R}^2}m\left( x \right) |^2}dx
 &\le \int_{Rr_{\epsilon}\le |x|\le \delta}{|\nabla _{\mathbb{R}^2}u_{\epsilon}^{*}\left( x \right) |^2}dx
\\
&\le \int_{Rr_{\epsilon}\le |x|\le \delta}{|\nabla _{\mathbb{R}^2}\bar{u}_{\epsilon}\left( x \right) |^2}dx
\\
&\le 1-\int_{\Sigma \backslash \phi ^{-1}\left( \mathbb{B}_{\delta} \right)}{\left| \nabla_g u_{\epsilon} \right|}^2dv_g-\int_{\phi ^{-1}\left( \mathbb{B}_{Rr_{\epsilon}} \right)}{\left| \nabla_g u_{\epsilon} \right|}^2dv_g.
  \end{aligned}
\end{equation}
Now we compute $\int_{\Sigma \backslash \phi ^{-1}\left( \mathbb{B}_{\delta} \right)}{\left| \nabla_g u_{\epsilon} \right|}^2dv_g$ and $\int_{\phi ^{-1}\left( \mathbb{B}_{Rr_{\epsilon}} \right)}{\left| \nabla_g u_{\epsilon} \right|}^2dv_g$.
Integration by parts and (\ref{G}) lead to
\begin{equation*}
\int_{\Sigma \backslash \phi ^{-1}\left( \mathbb{B}_{\delta} \right)}{|}\nabla_g G|^2dv_g=-\frac{1}{2\pi}\log \delta +A_{x_0}+\alpha \lVert G \rVert _{2}^{2}+o_{\epsilon}\left( 1 \right) +o_{\delta}\left( 1 \right).
\end{equation*}
According to (\ref{phi}), (\ref{6}) and (\ref{7}), one gets
\begin{equation}\label{22}
\int_{\phi ^{-1}\left( \mathbb{B}_{Rr_\epsilon} \right)}{\left| \nabla_g u_{\epsilon} \right|}^2dv_g=\frac{1}{c_{\epsilon}^{2}}\left( \frac{1}{4\pi}\log ( 1+\pi R^2  ) -\frac{1}{4\pi}+o_\epsilon\left( 1 \right)+o_{R}\left( 1 \right)\right),
\end{equation}
where $o_R\left( 1 \right) \rightarrow 0$ as $R\rightarrow +\infty $.
In view of (\ref{17})-(\ref{22}), we have
\begin{equation*}
\log \frac{\lambda _{\epsilon}}{c_{\epsilon}^{2}}\leq \log {\pi}+1+4\pi A_{x_0} +o(1),
\end{equation*}
where $o\left( 1 \right) \rightarrow 0$ as $\epsilon \rightarrow 0$ first, then $R\rightarrow +\infty $ and $\delta \rightarrow 0$.
Hence (\ref{16}) is followed.
From (\ref{16}) and Lemma \ref{L6}, we lead to the lemma.
	\end{proof}
The proof of ($ii$) follows immediately from Lemma \ref{L8} under the hypothesis of $c_{\epsilon}\rightarrow +\infty$.
When $|c_{\epsilon}|\leq C$, using Lebesgue's dominated convergence, we get
	\begin{equation}\label{18}
J _{4\pi}^\alpha(u_0)=\sup_{u\in \mathcal{H}}J _{4\pi}^\alpha(u).
	\end{equation}
And it is easy to see $u_0\in\mathcal{H}$.
Therefore, Part ($ii$) holds.

\section{Proof of Part ($iii$)}

The content in this section is carried out under the hypotheses $0 \leq \alpha < \lambda_{ p}$, $c_{\epsilon}\rightarrow +\infty$ and $x _ { \epsilon } \rightarrow x _ { 0 } \in \Sigma$.
Set a cut-off function
$\xi \in C _ { 0 } ^ { \infty } \left( B _ { 2 R \epsilon } ( x_0 ) \right)$ with $\xi = 1$ on $B _ { R \epsilon } ( x_0)$ and $\| \nabla \xi \| _ { L ^ { \infty } } = O (  (  R \epsilon )^{-1} )$. Denote $\tau =G+(2\pi)^{-1}\log r-A_{x_0}$, where $G$ is defined as (\ref{G}). Let $R= -\log \epsilon $, then $R\rightarrow+\infty$ and $R\epsilon\rightarrow0$ as $\epsilon\rightarrow0$.
We construct a blow-up sequence
	\begin{equation}\label{23}
    v_{\epsilon}=\left\{\begin{aligned}
	&\frac{c^2-\frac{\log ( 1+\frac{\pi r^2}{\epsilon ^2} )}{4\pi} +B}{\sqrt{c^2+\alpha \lVert G \rVert _{p}^{2}}},&		\,\,& r\le R\epsilon,\\
	&\frac{G-\xi \tau}{\sqrt{c^2+\alpha \lVert G \rVert _{p}^{2}}},&		\,\,\ \ &R\epsilon <r<2R\epsilon,\\
	&\frac{G}{\sqrt{c^2+\alpha \lVert G \rVert _{p}^{2}}},&		\,\,&r\ge 2R\epsilon,\\
\end{aligned} \right.
	\end{equation}
where $r=dist(x,x_0)\geq0$ and $B$, $c$ are constants to be determined later.
In order to assure that $v _{\epsilon}\in \mathcal{H} \cap C^{\infty}( \Sigma)$, we obtain
$$c^2-\frac{1}{4\pi}\log ( 1+\pi R^2 ) +B=-\frac{1}{2\pi}\log ( R\epsilon) +A_{x_0},$$
that is to say,
	\begin{equation}\label{c1}
c^2=\frac{1}{4\pi}\log \pi -B-\frac{1}{2\pi}\log \epsilon +A_{x_0}+O(R^{-2}).
	\end{equation}
From $\lVert \nabla_g v _{\epsilon} \rVert _2=1$, there holds
\begin{equation}\label{c2}
{c}^2=A_{x_0}-\frac{\log \epsilon}{2\pi}+\frac{\log \pi}{4\pi}-\frac{1}{4\pi}+O(R^{-2}) +O( R\epsilon \log ( R\epsilon ) ) +o_{\epsilon}\left( 1 \right).
\end{equation}
According to (\ref{c1}) and (\ref{c2}), one gets
	\begin{equation}\label{B}
B = \frac { 1 } { 4 \pi } + O(R^{-2})+ O ( R \epsilon \log ( R \epsilon ) )+o_{\epsilon}\left( 1 \right).
	\end{equation}
A delicate and simple calculation shows
\begin{equation}\label{28}
  \lVert v_{\epsilon} \rVert _{p}^{2}=\frac{\lVert G \rVert _{p}^{2}+O\left( R\epsilon \log \left( R\epsilon \right) \right)}{c^2+\alpha \lVert G \rVert _{p}^{2}}
  \ge \frac{\lVert G \rVert _{p}^{2}+O( R\epsilon \log ( R\epsilon ))}{c^2}\left( 1-\frac{\alpha \lVert G \rVert _{p}^{2}}{c^2} \right) .
\end{equation}
It follows that on $\left(B _ { R \epsilon } \left( x _ { 0 } \right),\ g\right)$
$$4\pi v_{\epsilon}^{2}( 1+\alpha \lVert v_{\epsilon} \rVert _{p}^{2} ) \ge 4\pi c^2-\text{2}\log \left( 1+\pi \frac{r^2}{\epsilon ^2} \right) +8\pi B-\frac{4\pi \alpha ^2\lVert G \rVert _{p}^{4}}{c^2}+O\left( \frac{\log R}{c^4} \right).$$
Setting $ { v }^* _ { \epsilon } =   \int _ { \Sigma } v _ { \epsilon } dv_g/{ {\rm Area}( \Sigma) }$,
one gets $ { v }^* _ { \epsilon } = O ( ( R \epsilon ) ^ { 2 } \log \epsilon)$.
Hence by (\ref{c1})-(\ref{28}), there holds
	\begin{equation}\label{27}
\int_{B_{R\epsilon}( x_0 )}{e}^{4\pi ( v_{\epsilon}-v_{\epsilon}^{*} ) ^2( 1+\alpha \lVert v_{\epsilon}-v_{\epsilon}^{*} \rVert _{2}^{2} )}\,dv_g\ge \pi e^{1+4\pi A_{x_0}}-\frac{4\pi \alpha ^2\lVert G \rVert _{p}^{4}}{c^2}e^{1+4\pi A_{x_0}}+O\left( \frac{\log R}{c^4} \right) +O\left( \frac{\log\log \epsilon}{R^2} \right).
	\end{equation}	
On the other hand, from the fact of $e^t\geq t+1$ for any $t>0$ and (\ref{23}), we get
	\begin{equation}\label{24}
	\begin{aligned}
\int_{\Sigma \backslash B_{R\epsilon}( x_0)}{e}^{4\pi ( v_{\epsilon}-v_{\epsilon}^{*} ) ^2( 1+\alpha \lVert v_{\epsilon}-v_{\epsilon}^{*} \rVert _{2}^{2})}\,dv_g
&\ge \int_{\Sigma \backslash B_{2R\varepsilon}( x_0 )}{( 1+4\pi ( v_{\epsilon}-v_{\epsilon}^{*} ) ^2)}\,dv_g\\
    &\ge \quad {\rm Area}\left( \Sigma \right) +\frac{4\pi \lVert G \rVert _{2}^{4}}{c^2}+O\left( \frac{\log R}{c^4} \right) +O\left( R\epsilon \right).
	\end{aligned}
	\end{equation}
From (\ref{27}) and (\ref{24}), there holds
	\begin{equation*}\label{25}
    \begin{aligned}
\int_{\Sigma}{e}^{4\pi ( v_{\epsilon}-v_{\epsilon}^{*} ) ^2( 1+\alpha \lVert v_{\epsilon}-v_{\epsilon}^{*} \rVert _{2}^{2} )}\,dv_g\ge
& {\rm Area}\,( \Sigma ) +\pi e^{1+4\pi A_{x_0}}+\frac{4\pi \lVert G \rVert _{2}^{2}}{c^2}\left( 1-\frac{\pi \alpha ^2\lVert G \rVert _{p}^{4}}{\lVert G \rVert _{2}^{2}}e^{1+4\pi A_{x_0}} \right)\\
& +O\left( \frac{\log\log \epsilon}{R^2} \right) +O\left( \frac{\log R}{c^4} \right) +O\left( R\epsilon \right).
    \end{aligned}	
    \end{equation*}
    According to $R= -\log \epsilon $ and (\ref{c1}), we obtain
	\begin{equation}\label{26}
J _ { 4 \pi } ^ { \alpha } ( v _ { \epsilon } - v^*_ { \epsilon } ) > {\rm Area}\,(\Sigma) + \pi e ^ { 1 + 4 \pi A _ {x_0} }
    \end{equation}
for sufficiently small $\alpha$ and $\epsilon$.
The contradiction between (\ref{3}) and (\ref{26}) indicates that $c_{\epsilon}$ must be bounded when $\alpha$ is sufficiently small.
Obviously, $u_0$ is the extremal function under the assumption that $c_\epsilon$ is bounded from (\ref{18}).
Therefore, we lead to Part ($iii$).
$\hfill\Box$

\end{document}